\documentclass[12pt]{article}
\usepackage{amsmath,amsfonts,amssymb}
\usepackage{alltt}

\usepackage{amsmath,amsfonts,ifthen,fullpage,enumerate}  

\usepackage{amsmath,amsfonts,amssymb}

\usepackage{makecell}  
\usepackage{float}
\usepackage{longtable}

\usepackage{graphicx}        

\usepackage[table]{xcolor}   

\def\cSU{\mathop{\cal SU}\nolimits}

\def\pmat#1{\begin{pmatrix}#1\end{pmatrix}}

\def\question#1{{\bf Question: }#1}
\def\question#1{}

\def\cD{{\cal D}}
\def\cT{{\cal T}}
\def\cI{{\cal I}}
\def\cO{{\cal O}}

\def\R{\mathbb{R}}

\def\CC{\mathbb{C}}

\def\FF{\mathbb{F}}

\def\ZZ{\mathbb{Z}}

\def\Cd{\C^d}
\def\Rd{\R^d}

\def\C{\mathbb{C}}

\def\SS{\mathbb{S}}

\newcommand{\RR}{\mathbb{R}}

\newtheorem{theorem}{Theorem}[section]

\newtheorem{lemma}{Lemma}[section]
\newtheorem{example}{Example}[section]

\newenvironment{proof}{{\noindent \it
Proof.}}{\hfill$\Box$\medskip}
\newif\ifdraft\def\draft{\drafttrue\hoffset=.8truecm\showlabeltrue
\def\comment##1{{\bf comment: ##1}}
\headline={\sevenrm \hfill \ifx\filenamed\undefined\jobname\else\filenamed\fi%
(.tex) (as of \ifx\updated\undefined???\else\updated\fi)
 \TeX'ed at {\hour\time\divide\hour by 60{}%
\minutes\hour\multiply\minutes by 60{}%
\advance\time by -\minutes
\the\hour:\ifnum\time<10{}0\fi\the\time\  on \today\hfill}}
}

\def\inpro#1{\langle#1\rangle}
\def\ip#1{\langle\kern-.28em\langle#1\rangle\kern-.28em\rangle_\nu}

\def\cU{{\cal U}}

\def\cH{{\cal H}}

\def\norm#1{\Vert#1\Vert}
\def\openR{{{\rm I}\kern-.16em {\rm R}}}

\def\Fd{\FF^d}
\let\ga\alpha

\let\gb\beta

\let\gs\sigma

\let\go\omega
\let\gO\Omega
\let\ga\alpha

\let\gb\beta

\let\gs\sigma
\def\inpro#1{\langle#1\rangle}

\def\Hom{\mathop{\rm Hom}\nolimits}

\def\formeq{\the\sectionno.\the\equationno}  
\def\elabel#1/#2/#3/{\global\advance\equationno by 1 %
\ifx#1\empty\else\emember#1%
\ifshowlabel\marginal{\string#1}\fi\fi%
\ifmmode\eqno{#3(\formeq#2)}\else#3\formeq#2\fi} 

\def\makeblanksquare#1#2{
\dimen0=#1pt\advance\dimen0 by -#2pt
      \vrule height#1pt width#2pt depth0pt\kern-#2pt
      \vrule height#1pt width#1pt depth-\dimen0 \kern-#1pt
      \vrule height#2pt width#1pt depth0pt \kern-#2pt
      \vrule height#1pt width#2pt depth0pt
}

\title{\bf 
Constructing high order spherical designs as a union of 
two of lower order
}

\author{
Mozhgan Mohammadpour, Shayne Waldron\\
 \\
Department of Mathematics \\ University of Auckland\\
Private
Bag 92019, Auckland, New Zealand\\
e-mail: waldron@math.auckland.ac.nz}
\date{\today}

\begin{document}

\maketitle 

\begin{abstract}
We show how the variational characterisation of spherical designs 
can be used to take a union of spherical designs to obtain a spherical 
design of higher order (degree, precision, exactness) with a small
number of points.
The examples that we consider involve
taking the orbits of two vectors under the 
action of a complex reflection group to obtain 
a weighted spherical $(t,t)$-design. 
These designs have a high degree of symmetry
(compared to the number of points), and many are 
the first known construction of such a design,
e.g.,
a $32$ point $(9,9)$-design for $\CC^2$,
a $48$ point $(4,4)$-design for $\CC^3$, 
and a $400$ point $(5,5)$-design for $\CC^4$.
From a real reflection group, we construct a $360$ point $(9,9)$-design
for $\RR^4$ (spherical half-design of order $18$), 
i.e., a $720$ point spherical $19$-design for $\RR^4$.
\end{abstract}

\bigskip
\vfill


 \ifdraft (Other designs, configurations) \else\fi
42C15, \ifdraft General harmonic expansions, frames  \else\fi
65D30; \ifdraft (Numerical integration) \else\fi
\quad
secondary
94A12. \ifdraft (Signal theory [characterization, reconstruction, etc.]) \else\fi

\vskip .5 truecm
\hrule
\newpage

\section{Introduction}

Let $\SS$ be the unit sphere in $\Rd$ or $\Cd$, and $\gs$ be 
normalised surface area measure on $\SS$.
A {\bf weighted spherical design} is a finite set (or sequence) of points $X$ in $\SS$ 
and {\bf weights} $w_x\in\RR$, $x\in X$,
for which the integration (cubature) rule
\begin{equation}
\label{P-design-def}
\int_\SS f \,d\gs = \sum_{x\in X} w_x f(x), \qquad \forall f\in P,
\end{equation}
holds for some finite dimensional space of functions $P$ defined on $\SS$
(usually a unitarily invariant polynomial space).
Such configurations of points are known to exist for every choice of $P$
(see \cite{BT06}, \cite{SZ84}).
For 
certain choices there is great interest in explicit constructions,
especially those with a minimal number of points, e.g., the ``tight spherical
designs'' of algebraic combinatorics \cite{BB09}.
The optimal configurations often have a high degree of symmetry, and are
closely related to optimal spherical packings \cite{MP19}, \cite{JKM19},
\cite{V17},
and points minimising a potential function on the sphere
\cite{BGMPV19}.


If $X$ and $Y$ are weighted spherical designs with weights $(w_x^X)$ and
$(w_y^Y)$, then for any fixed $\ga\in\RR$ and $f\in P$, we have
$$ \sum_{x\in X} (\ga w_x^X) f(x)
+\sum_{y\in Y} ((1-\ga) w_y^Y) f(y)
= \ga \int_\SS f \,d\gs + (1-\ga)\int_\SS f \,d\gs 
= \int_\SS f \,d\gs, $$
so that $X\cup Y$ is a weighted spherical design, 
with the ``affine combination'' of the 
weights 
\begin{equation}
\label{wXYdefn}
w_a^{X\cup Y,(\ga,1-\ga)} :=
\begin{cases}
\ga w_x^X, & a=x\in X; \cr
(1-\ga) w_y^Y, & a=y\in Y.
\end{cases}
\end{equation}
The weights of a spherical design are usually taken to be
positive, and so it would be natural to take a ``convex combination''
of the weights, i.e., to choose $0<\ga,1-\ga<1$.
We will call $(\gb_X,\gb_Y)=(\ga,1-\ga)$ the {\bf weighting} of
the union $X\cup Y$. It is usually assumed the weights add to $1$
(this follows if $P$ contains the constants), in which case
$$ \gb_X=\sum_{x\in X} w_x^{X\cup Y,(\ga,1-\ga)}, \qquad
\gb_Y=\sum_{y\in Y} w_y^{X\cup Y,(\ga,1-\ga)}, \qquad \gb_X+\gb_Y=1. $$

The purpose of this paper is to try and choose the weighting of a union of 
spherical designs to obtain one of {\em higher order}, i.e., for 
which the space $P$ in (\ref{P-design-def}) is enlarged.
If one were to try and use (\ref{P-design-def}) to do this, then 
one could increase 
$P$ by just one dimension, by solving an appropriate linear equation for $\ga$.

When $P$ is a unitarily invariant space of polynomials,
(\ref{P-design-def}) can be
replaced by a single {\it quadratic equation} in the weights $(w_x)_{x\in X}$
with coefficients involving just (the inner products between) the points $X$,
which comes from a variational characterisation \cite{W19}. By considering
this quadratic for the union of designs $X$ and $Y$, and a 
unitarily invariant space $Q$, 
it follows that:

\begin{lemma}
\label{quadraticexistlemma}
There is a 
quadratic equation in $\ga=\gb_X$, which 
if solvable,
 gives a weighting for
the union of spherical designs $X$ and $Y$ for $P$ to be one for 
a larger space $Q$.
\end{lemma}

This is useful only if one can choose $X$, $Y$
and $Q$ (large enough to be of interest), 
so that the quadratic equation has a real root, preferably with $0<\ga<1$.
Remarkably, we show that this approach actually works quite successfully.
We will primarily consider the class of (complex) spherical $(t,t)$-designs.
The basic properties in the milieu are:

\begin{itemize}
\item $X$ and $Y$ are chosen to have a small number of points. 
In practice, this means that they are an orbit of a unitary 
group action, with a large stabiliser.
\item $X$ and $Y$ must have the right relationship.
Clearly, 
we cannot take $Y=X$ and gain anything more. One could take
$Y=UX$ with $U$ unitary, but this adds additional parameters
to the quadratic (making it more likely to find one which is solvable, 
but less tractable). Here we take $X$ and $Y$ to be orbits of the same
group action.
\item $Q$ must be large enough to be of interest, but not so large
that the quadratic has no real roots. In practice, $P$ is polynomials
up to some degree, and we take $Q$ to be the same space for
polynomials one degree larger.
\end{itemize}
Our constructions for orbits of finite complex 
reflection groups are summarised in \S \ref{summarysect}.

\section{Spherical $(t,t)$-designs and half-designs of order $2t$}

For $t=1,2,\ldots$, 
every finite set of vectors $X$ in $\Cd$ satisfies the inequality
\begin{equation}
\label{Welchbound}
\sum_{x\in X}\sum_{y\in X} |\inpro{x,y}|^{2t} 
\ge c_t(\Cd) \Bigl(\sum_{x\in X}\norm{x}^{2t}\Bigr)^2,
\qquad c_t(\Cd):={1\over{t+d-1\choose t}}.
\end{equation}
A set of nonzero vectors giving equality in (\ref{Welchbound}) is
called a {\bf spherical $(t,t)$-design}.
A spherical $(t,t)$-design $X$ is a weighted spherical design for
the complex sphere \cite{W17},
where $x\in X$ corresponds to $\hat x:={x\over\norm{x}}\in\SS$,
 and the weights and polynomial space are
\begin{equation}
\label{weightsandP}
w_x = {\norm{x}^{2t}\over\sum_{a\in X}\norm{a}^{2t}}, 
\qquad
P=\Hom(t,t). 
\end{equation}
Here $\Hom(p,q)$ is the space of homogeneous polynomials in the variables $z\in\Cd$ and
$\overline{z}$ which are of degree $p$ in $z$ and degree $q$ in $\overline{z}$.
The variational characterisation is
\begin{equation}
\label{varchar}
\sum_{x\in X}\sum_{y\in X} w_x w_y |\inpro{\hat x,\hat y}|^{2t} = c_t(\Cd).
\end{equation}
From (\ref{Welchbound}), it follows that a spherical $(t,t)$-design is a 
projective object, i.e., multiplying a point $x\in X$ by a unit scalar
gives another such design, and so $x$ can be identified with the 
complex line through $x$ and the origin. When a spherical $(t,t)$-design
is viewed as a collection of lines, then the term 
{\it weighted complex projective $t$-design} is also used \cite{RS07}. 
Notable examples include {\it tight frames} which are the $(1,1)$-designs
\cite{W18} (those with the minimal number being the orthogonal bases), 
and {\it SICs} (sets of $d^2$ equiangular lines in $\Cd$) 
which are $(2,2)$-designs with the minimal number of vectors
\cite{ACFW18}.

It is not 
obvious from the definition that unions
of spherical $(t,t)$-designs are again spherical $(t,t)$-designs. 
This follows from the spherical design property (\ref{P-design-def}).

\begin{theorem}
\label{unionsof(t,t)-designs}
If $X$ and $Y$ are are spherical $(t,t)$-designs,
then so is any convex union of them, such as $X\cup Y$, and in
particular
\begin{equation}
\label{suprisingequality}
\sum_{x\in X}\sum_{y\in Y} |\inpro{x,y}|^{2t}
= c_t(\Cd) \Bigl(\sum_{x\in X}\norm{x}^{2t}\Bigr)
\Bigl(\sum_{y\in Y}\norm{y}^{2t}\Bigr).
\end{equation}
\end{theorem}

\begin{proof}
The union $X\cup Y$, with weights given by (\ref{weightsandP}),
is given by the weighting
$$ \gb_X = {\sum_{a\in X}\norm{a}^{2t}
\over \sum_{x\in X}\norm{x}^{2t} + \sum_{y\in Y}\norm{y}^{2t}}, \qquad
\gb_Y = {\sum_{b\in Y}\norm{b}^{2t}
\over \sum_{x\in X}\norm{x}^{2t} + \sum_{y\in Y}\norm{y}^{2t}}. $$
Eliminating terms for equality in (\ref{Welchbound})
for $X$, $Y$ and $X\cup Y$ gives (\ref{suprisingequality}).
\end{proof}

Let $X$ and $Y$ be finite subsets of $\SS$ and 
$(w_x^X)$ and $(w_y^Y)$ be corresponding weights.
By the variational characterisation (\ref{varchar}),
their union $X\cup Y$ with the weighting $(\gb_X,\gb_Y)=(\ga,1-\ga)$ is
a spherical $(t,t)$-design if and only if $\ga$ satisfies
$$ \sum_{a\in X\cup Y} \sum_{b\in X\cup Y} 
w_a^{X\cup Y,(\ga,1-\ga)} w_b^{X\cup Y,(\ga,1-\ga)}
|\inpro{a,b}|^{2t} = c_t(\Cd), $$
which, by (\ref{wXYdefn}), expands to 
the following quadratic equation in $\ga$
\begin{align}
\ga^2 \sum_{a\in X}\sum_{b\in X} w_a^X w_b^X |\inpro{a,b}|^{2t} 
+& (1-\ga)^2 \sum_{a\in Y}\sum_{b\in Y} w_a^Y w_b^Y |\inpro{a,b}|^{2t}
\nonumber \\
+&2 \ga(1-\ga) \sum_{a\in X}\sum_{b\in Y} w_a^X w_b^Y |\inpro{a,b}|^{2t} 
= c_t(\Cd). 
\label{exampleofquadratic}
\end{align}
This is an instance of Lemma \ref{quadraticexistlemma}. 
Here (and in general) the coefficients of the quadratic depend only on
the weights 
and the inner products between 
the elements of $X\cup Y$.

We find in convenient to use the {\bf normalised weights} 
$$ \hat w_x^X:=|X|w_x^X, $$ 
so that the normalised weights for $X$ add to $|X|$,
and they equal $1$ when they are all the same.
We now suppose the weights for $X$ and $Y$ are both constant
(as will be the case for an orbit under a unitary action),
so that the normalised weights for $X\cup Y$ have the form
$$ (|X|+|Y|)w_a^{X\cup Y,(\ga,1-\ga)} =: 
\begin{cases}
\hat w_X, & a\in X; \\
\hat w_Y, & a\in Y.
\end{cases} $$
Since $|X|\hat w_X +|Y|\hat w_Y = |X|+|Y|$,
  for $\hat w_X,\hat w_Y\ne0$,
it follows from (\ref{wXYdefn}) that $X\cup Y$
with the weighting given by $z=\hat w_X$ 
is a spherical $(t,t)$-design if and only if it 
\begin{align} 
z^2 \sum_{a\in X}\sum_{b\in X} |\inpro{a,b}|^{2t} 
+& \left({|X|+|Y|-|X|z\over|Y|}\right)^2 \sum_{a\in Y}\sum_{b\in Y} 
|\inpro{a,b}|^{2t} \nonumber \\
&\hskip-1.5truecm +2 z \left({|X|+|Y|-|X|z\over|Y|}\right) \sum_{a\in X}\sum_{b\in Y} 
 |\inpro{a,b}|^{2t} 
= (|X|+|Y|)^2 c_t(\Cd). 
\label{z-quadratic}
\end{align}
Once a suitable $\hat w_X$ has been found, 
the other parameters can then be calculated from
\begin{equation}
\label{otherpars}
\hat w_Y = {|X|+|Y|-|X|\hat w_X\over|Y|}, 
\qquad \gb_X = {|X|\hat w_X\over|X|+|Y|}, \quad
\gb_Y = {|Y|\hat w_Y\over|X|+|Y|}.
\end{equation}
 
\vfil\eject

For vectors $X$ in $\Rd$, 
the following sharpening of (\ref{Welchbound}) 
is possible (see \cite{W17})
\begin{equation}
\label{realWelchbound}
\sum_{x\in X}\sum_{y\in X} |\inpro{x,y}|^{2t} 
\ge c_t(\Rd) \Bigl(\sum_{x\in X}\norm{x}^{2t}\Bigr)^2,
\qquad c_t(\Rd):={1\cdot3\cdot5\cdots(2t-1)\over d(d+2)\cdots (d+2(t-1))}.
\end{equation}
The corresponding spherical designs are called {\bf spherical 
half-designs} \cite{KP11}. They integrate $P=\Hom(2t)$, the space of 
homogeneous polynomials of degree $2t$ on $\Rd$,
and are characterised by 
\begin{equation}
\label{realvarchar}
\sum_{x\in X}\sum_{y\in X} w_x w_y |\inpro{\hat x,\hat y}|^{2t} = c_t(\Rd).
\end{equation}
Our previous discussion on spherical $(t,t)$-designs extends to 
spherical half-designs in the obvious way, i.e., replace 
$c_t(\Cd)$ by $c_t(\Rd)$. We will not labour the point, with a
spherical $(t,t)$-design for $\Rd$ understood to be a spherical
half-design of order $2t$.

\section{Highly symmetric tight frames and reflection groups}

Since $\norm{\cdot}^2=\inpro{\cdot,\cdot}$ is constant on the sphere, 
the space $P=\Hom(t,t)$ in (\ref{P-design-def})
integrated by a spherical $(t,t)$-design satisfies
$$\Hom(p-1,q-1)|_\SS \subset \Hom(p,q)|_\SS.$$
Hence a spherical $(t,t)$-design is a spherical $(r,r)$-design for $r=0,1,\ldots, t$. 
In particular, its weights add to $1$ ($r=0$) and it is a tight frame ($r=1$) 
for $t\ge1$. The analogous result for 
spherical half-designs of order $2t$ follows from the fact
$\Hom(2(t-1))|_\SS\subset \Hom(2t)|_\SS$.

The following notion of a ``highly symmetric'' tight frame was 
given in \cite{BW13}.
\begin{itemize}
\item[] 
A finite frame of distinct vectors is {\bf highly symmetric} if the 
action of its symmetry group is irreducible, transitive, and the 
stabiliser of any one vector (and hence all) is a nontrivial subgroup
which fixes a space of dimension exactly one.
\end{itemize}
The upshot of this definition, is that for every unitary 
irreducible representation
of a finite group on $\Rd$ or $\Cd$, there is a finite (possibly empty) set of highly symmetric tight frames
(up to unitary equivalence) given as a group orbit, which has a 
nontrivial stabiliser (the number of vectors is less than the order 
of the group). In theory, these highly symmetric tight frames can be 
calculated for a given group (or representation), and this was done
primarily in the case of finite complex reflection groups in \cite{BW13}.

A finite group of linear transformations on $\Rd$ or $\Cd$ is 
a {\bf complex reflection group} if it is generated by 
complex reflections, i.e., transformations which fix a hyperplane
(and have finite order).
The finite irreducible complex reflection groups were classified by 
Shephard and Todd (see \cite{ST54}, \cite{LT09}). There are three
infinite families of imprimitive reflection groups 
of the type $G(m,p,n)$, $p\vert m$,
and $31$ primitive complex reflection groups
$G_4,\ldots,G_{34}$
 in dimensions $2,3,\ldots,8$, which are referred to as the
{\it Shephard-Todd groups} 
 with numbers $4,5,\ldots,34$. 
The complex reflection groups are a generalisation of 
the real reflection groups (classified by Coxeter). 
The Shephard-Todd classification
contains the real reflection groups (numbers $23$, $28$, $30$, $35$, $36$ and $37$). 
In many presentations, the
generators of the real reflection groups are given as matrices over a cyclotomic field.

The highly symmetric tight frames 
for the Shephard-Todd groups were calculated in \cite{BW13}. 
Their strength as $(t,t)$-designs (the largest $t$ can be)
was calculated in \cite{HW18} 
by utilising 
magma software  
of Don Taylor 
to calculate the {\it maximal parabolic subgroups} 
(which stablise the vectors of a highly symmetric tight frame).
Later, it was shown that in most, but not all cases,
the strength of such a design was shared by all orbits
(where the action is unitary),
and that it could be calculated from a  complex 
harmonic Molien-Poincar\'e series
\cite{RS14},
\cite{MW19}. The corresponding results for orthogonal actions on real spaces
were considered earlier by \cite{B79}, \cite{HP04}.
In both the real and complex cases, 
we will call this the {\it generic} strength of
an orbit.

\begin{example} If the unitary action of a finite group on 
$\Fd=\Rd,\Cd$ is irreducible, i.e., every orbit of every
nonzero vector spans $\Fd$, then every orbit of a nonzero 
vector is a tight frame, i.e., a $(1,1)$-design (this is 
equivalent to the action being irreducible). Hence the
generic strength of an orbit of an irreducible complex reflection 
group is at least $t=1$.
\end{example}

Our main result is the proof of concept: 
$$ \hbox{\it The quadratic (\ref{z-quadratic}) can be solved
to find a union of spherical designs with higher order.} $$
A summary of our calculations for the highly symmetric
spherical $(t,t)$-designs for the complex reflection groups
is given in Section \ref{summarysect}. 
Combining these 
gives the following:

\begin{theorem} Let $G$ be a primitive irreducible complex
reflection group (these have Shephard-Todd numbers 4--34). 
If $X$ and $Y$ are different highly symmetric tight frames for $G$,
then there is unique rational weighting for which $X\cup Y$
is a spherical $(t,t)$-design, where $t$ is strictly larger
than that of a generic orbit.
Moreover, for every case where there are two or more highly
symmetric tight frames, a pair can be chosen for which the
weighting is convex, i.e., has positive entries. 
\end{theorem}

In Section \ref{quadstructsect},
we give evidence to suggest that such a result also holds 
for {\it any} pair of orbits, i.e., the fact that the orbit
is highly symmetric is important only in that its size is small.

We finish this section with some technical comments about our calculations.


\begin{itemize}
\item Our calculations were done in magma, using the
software {\tt Complements.m} of Don Taylor to calculate 
the maximal parabolic subgroups. 
Magma writes vectors as rows, and the action of
a matrix group, e.g., in {\tt Eigenspace}, is by right multiplication,
and so
our code must be read with this in mind.
\item For an orbit of a unit vector to lie on the sphere, the
group action must be unitary. The presentations
of the complex reflection groups (or more generally irreducible
representations)
provided in magma are not all unitary. One way around this, 
is to consider the canonical Gramian (which can be
calculated from the Gramian) of the orbit of the 
nonunitary representation \cite{W18}. This can be done, but becomes 
unfeasible eventually. Another way, is to find a 
Hermitian matrix which gives the quadratic form under which
the action is unitary (as was done in \cite{BW13}). 
This works better for large examples,
as the inner products in sums such as 
(\ref{z-quadratic})
can be created and added to the sum
one by one. Thus for orbits of large size there is no need to
create the Gramian.
\end{itemize}

\section{The structure of the quadratic}
\label{quadstructsect}

For weighted sets $X$ and $Y$ of points on the sphere, let
$$ b_{XY}^{(t)} := \sum_{a\in X}\sum_{b\in Y} w_a^X w_b^Y
|\inpro{a,b}|^{2t}. $$
If $X$ and $Y$ are spherical $(t,t)$-designs for $\Fd$, 
then by Theorem \ref{unionsof(t,t)-designs}, we have
$$ b_{XY}^{(t)} = c_t(\Fd), $$
so that 
$$ b_{XX}^{(t)}b_{YY}^{(t)}- (b_{XY}^{(t)})^2\ne0, \qquad
b_{XX}^{(t)}+b_{YY}^{(t)}-2b_{XY}^{(t)}\ne0, $$ 
when $X$ and $Y$
are not both spherical $(t,t)$-designs.

It seems that in the many cases considered so far, when there
is a root of (\ref{exampleofquadratic})
for a union of lower order designs,
then the root is a double root, i.e., the discriminant is zero
\begin{equation}
\label{discrimzero}
b_{XX}^{(t)}b_{YY}^{(t)}- (b_{XY}^{(t)})^2
=c_t(\Cd) \left(b_{XX}^{(t)}+b_{YY}^{(t)}-2b_{XY}^{(t)}\right),
\end{equation}
and we have the simple formula
$$ \gb_X 
= {b_{YY}^{(t)}-b_{XY}^{(t)}\over b_{XX}^{(t)}+b_{YY}^{(t)}-2b_{XY}^{(t)} },
\qquad \gb_Y 
= {b_{XX}^{(t)}-b_{XY}^{(t)}\over b_{XX}^{(t)}+b_{YY}^{(t)}-2b_{XY}^{(t)} }.
$$
This seems to hold for any pair of orbits, i.e., it has nothing to do with
it being a highly symmetric tight frame. 
Suppose that there is a unitary action of $G$ on $\Fd$, and let 
\begin{equation}
\label{pGtdef}
p_G^{(t)}(x,y):={1\over|G|}\sum_{g\in G} |\inpro{x,gy}|^{2t}
={1\over|G|^2}\sum_{g\in G}\sum_{h\in G} |\inpro{gx,hy}|^{2t}
= b_{Gx,Gy}^{(t)},
\end{equation}
where $Gx:=(gx)_{g\in G}$.
Then the condition for there to be a unique weighting for which 
the union of the orbits of $x$ and $y$ is
a spherical $(t,t)$-design is that
$$ 
p_G^{(t)}(\hat x,\hat x)p_G^{(t)}(\hat y,\hat y)
-\bigl(p_G^{(t)}(\hat x,\hat y)\bigr)^2\ne0,
\qquad
 \hbox{(the orbits are not both $(t,t)$-designs)} $$
where $\hat x:={x\over\norm{x}}$, and
$$ p_G^{(t)}(\hat x,\hat x)p_G^{(t)}(\hat y,\hat y)
-\bigl(p_G^{(t)}(\hat x,\hat y)\bigr)^2
= c_t(\Fd)\bigl( p_G^{(t)}(\hat x,\hat x) +p_G^{(t)}(\hat y,\hat y)
-2p_G^{(t)}(\hat x,\hat y)\bigr). $$
This condition can be written in terms of polynomials:

\begin{theorem}
\label{two-orbits-thm}
 (Two orbits) Let $G$ be a finite group with a unitary action on
$\Fd=\Rd,\Cd$. Then every generic pair of orbits has a unique
weighting which is a spherical $(t,t)$-design if and only if
the polynomial $f_G^{(t)}=f_{G,\FF}^{(t)} :\Fd\times\Fd\to\FF$ given by
\begin{equation}
\label{fgtdefn}
f_G^{(t)}(x,y)
:=p_G^{(t)}(x,x)p_G^{(t)}(y,y) -(p_G^{(t)}(x,y))^2
\end{equation}
is not identically zero, and
\begin{equation}
\label{pairsgivedesigns}
f_G^{(t)}(x,y)
= c_t(\Fd)\left( \norm{y}^{4t}p_G^{(t)}(x,x) +\norm{x}^{4t}p_G^{(t)}(y,y)
-2\norm{x}^{2t}\norm{y}^{2t}p_G^{(t)}(x,y)\right).
\end{equation}
where $p_G^{(t)}$ is given by (\ref{pGtdef}).
\end{theorem}

\begin{proof} 
Use $p_G^{(t)}(\hat x,\hat y) 
= {1\over\norm{x}^{2t}} {1\over\norm{y}^{2t}} p_G^{(t)}(x,y)$
to rewrite the previous conditions,
and then
multiply by $\norm{x}^{4t}\norm{y}^{4t}$. 
\end{proof}

Here the condition that the orbits $(gx)_{g\in G}$ and 
$(gy)_{g\in G}$ be generic is  
$f_G^{(t)}(x,y)\ne 0$.
Clearly, $f_G^{(t)}(x,y)=0$ if the orbits are equal or 
if both are spherical $(t,t)$-designs.
By way of comparison, the condition that every 
single orbit is a spherical $(t,t)$-design is that
$$ p_G^{(t)}(x,x)
= c_t(\Fd)\norm{x}^{4t}. $$
We will say that ``pairs of orbits give $(t,t)$-designs'', or similar,
if (\ref{pairsgivedesigns}) holds nontrivially.

Theorem \ref{two-orbits-thm} provides a computational way 
to verify when a generic pair of orbits has a unique
weighting giving a spherical $(t,t)$-design. We were able
to make this computation in magma for various groups $G$. 
Our preliminary results suggest:

\begin{changemargin}{0.70cm}{0.70cm} {\it 
Pairs of orbits give spherical $(t,t)$-designs
with $t$ higher than the generic strength,
for all complex reflection groups except the Coxeter group $D_4=G(2,2,4)$.
This also holds for many, but not all, irreducible representations.
}
\end{changemargin}

\noindent
The exact nature of these results is not yet clear, though it is 
related to
the irreducible unitarily invariant subspaces $H(p,q)$ of the
polynomials on $\Cd\cong\RR^{2d}$ (see \cite{R80}) that are 
integrated by the cubature rule for  a generic orbit.

Since the sum in (\ref{pGtdef}) is over all elements of the group $G$,
and cannot be simplified, e.g., by taking a transversal giving an
orbit of small size (as for highly symmetric tight frames) our 
calculations do not extend to all the groups considered 
in Section \ref{summarysect}.
 
We now give some selected examples.


\begin{example} Let $G$ be the dihedral group of order $6$
(a reflection group) generated by
$$ a=\pmat{-{1\over2}&-{\sqrt{3}\over2}\cr{\sqrt{3}\over2}&-{1\over2}}
\quad\hbox{\rm(rotation by ${2\pi\over3}$)},\qquad
b=\pmat{1&0\cr0&-1}\quad\hbox{\rm(reflection in the $x$-axis)}. $$
This is the first (faithful) irreducible group action in more than one dimension.

If $G$ acts on $\RR^2$, then every orbit is a $(2,2)$-design, 
so that 
$$ f_{G,\RR}^{(1)}=f_{G,\RR}^{(2)}=0, $$ 
and pairs of orbits give
$(t,t)$-designs for $\RR^2$ for $t=3,4,5$. Here
$$ f_{G,\RR}^{(3)}(x,y)
= 10\prod_{U\in\cU} (\inpro{x,Uy})^2,
$$
where $\cU$ is the set of unitary matrices
$$ \cU:=\left\{\pmat{0&1\cr-1&0},
\pmat{0&1\cr1&0},
\pmat{{\sqrt{3}\over2}&{1\over2}\cr{1\over2}&-{\sqrt{3}\over2}},
\pmat{{\sqrt{3}\over2}&{1\over2}\cr-{1\over2}&{\sqrt{3}\over2}},
\pmat{{\sqrt{3}\over2}&-{1\over2}\cr{1\over2}&{\sqrt{3}\over2}},
\pmat{{\sqrt{3}\over2}&-{1\over2}\cr-{1\over2}&-{\sqrt{3}\over2}}
\right\}, $$
and
$$ f_{G,\RR}^{(4)}(x,y)
 = \hbox{${7\over 4}$} \norm{x}^4\norm{y}^4 f_{G,\RR}^{(3)}(x,y),
\qquad f_{G,\RR}^{(5)}(x,y) =  (\hbox{${7\over 4}$} \norm{x}^4\norm{y}^4)^2
f_{G,\RR}^{(3)}(x,y). $$
It is not obvious from the definition (\ref{fgtdefn}) that 
these polynomials should be squares (or have common factors), 
or how the matrices in $\cU$ relate the elements of $G$.
If $G$ acts on $\CC^2$, then every orbit is a $(1,1)$-design, and pairs
of orbits give $(2,2)$-designs for $\CC^2$, where
$$ f_{G,\CC}^{(2)}(x,y)= 
\hbox{${1\over8}$}\bigl(\norm{x}^2\ga(y)+\norm{y}^2\gb(x)\bigr)^2
\bigl(\norm{x}^2\ga(y)+\norm{y}^2\gb(\overline{x})\bigr)^2, $$
with  
$$ \ga(y):=y_1\overline{y_2}-\overline{y_1}y_2, \qquad
\gb(x):=x_1\overline{x_2}-\overline{x_1}x_2.  $$
\end{example}

The lines in a spherical $(t,t)$-design for $\Cd$ which is an orbit depend only
the the matrices in the action group of the representation
up to unit scalar multiples.
Hence for the purpose of calculation, it suffices to take
a representative set of such matrices. A convenient way to
do this, is to take the associated group obtained by normalising
the matrices to have determinant $1$ (and taking all $d$ such
choices). This subgroup of $\cSU(\Cd)$ (as an abstract group)
was called a {\it canonical abstract error group} in \cite{CW17}.

The finite subgroups of $\cSU(\CC^2)$ are given by the ADE classification:
the binary tetrahedral, octahedral and icosahedral groups, together with
the binary dihedral groups $\cD_{2m}$ of order $4m$, which are generated
by the matrices
$$ a=\pmat{\go&0\cr0&\go}, \quad\go:=e^{2\pi i\over 2m},\quad
b=\pmat{0&-1\cr1&0}. $$
Except for $\cD_2\cong \ZZ_2\times\ZZ_2$, these are all irreducible
(see Theorem 5.14 of \cite{LT09} for details). A summary of our
calculations for these groups is given in Table \ref{2orbitdesigns2complexdims}.

\setlength{\tabcolsep}{3pt}
\renewcommand{\arraystretch}{1.2}
\begin{table}[H]
\caption{The unions of pairs of orbits for the
irreducible subgroups of $\cSU(\CC^2)$ (these correspond
to all irreducible representations). Here $t_{\rm generic}$ is
the strength of a generic orbit, and $t_{\rm pairs}$ is the range
of $t$ for which pairs of orbits give spherical $(t,t)$-designs.
 }
\begin{center}
\label{2orbitdesigns2complexdims}
\begin{tabular}{|p{4.9cm}| p{1.7cm}|p{1.3cm}||p{1.1cm}|p{1.1cm}|p{4cm}|}
\hline
Subgroup of $\cSU(\CC^2)$ & order & {\#}lines & $t_{\rm generic}$ & $t_{\rm pairs}$ & comments \\ \hline
Binary tetrahedral group $\cT$ & 24 & 12 & 1-2 & 3 & ST 4-7 (type $\cT$)  \\
Binary octahedral group $\cO$ & 48 & 24 & 1-3 & 4-5 & ST 8-15 (type $\cO$) \\
Binary icosahedral group $\cI$ & 120 & 60 & 1-5 & 6-9 & ST 16-22 (type $\cI$) \\
Binary dihedral group $\cD_4$ & 8 & 4 & 1 & $\{\}$ & associated real group  \\
Binary dihedral group $\cD_6$ & 12 & 6 & 1 & 2 & associated real group  \\
Binary dihedral group $\cD_{2m}$ & $4m\ge16$ & $2m$ & 1 & 2-3 & associated real group  \\
\hline
\end{tabular} 
\end{center}
\end{table}

The binary dihedral groups come from real representations,
and the corresponding pairs of real orbits (see Table \ref{2orbitdesigns2realdims})
give real spherical  $(t,t)$-designs. 
Let $D_{2m}=G(m,m,2)$ be the dihedral group of order $2m$
generated by  
$$ a=\pmat{\cos{2\pi\over m}&-\sin{2\pi\over m}\cr\sin{2\pi\over m}&\cos{2\pi\over m}},
\quad\hbox{\rm(rotation by ${2\pi\over m}$)},\qquad
b=\pmat{1&0\cr0&-1}\quad\hbox{\rm(reflection in the $x$-axis)}, $$
and $R_m=\inpro{a}$ be the rotation subgroup.
Since $b$ is a reflection, it does not have determinant $1$. Multiplying
it by the scalar $i$ gives a matrix in $\cSU(\CC^2)$. 
The subgroup $\inpro{a,ib}$ of $\cSU(\CC^2)$ is conjugate
to $\cD_{2m}$ for $m$ odd, 
and is conjugate to $\cD_m$ for $m$ even.

\setlength{\tabcolsep}{3pt}
\renewcommand{\arraystretch}{1.2}
\begin{table}[H]
\begin{center}
\caption{The unions of pairs of orbits for 
the irreducible subgroups of $\cO(\RR^2)$ (these correspond
to all irreducible representions).  }
\label{2orbitdesigns2realdims}
\begin{tabular}{|p{3.9cm}| p{2.4cm}|p{1.3cm}||p{2.1cm}|p{2.6cm}|p{2cm}|}
\hline
Subgroup of $\cO(\RR^2)$ & order & {\#}lines & $t_{\rm generic}$ & $t_{\rm pairs}$ & comments \\ \hline
Dihedral group $D_{2m}$ & $2m$ ($m$ odd) & $2m$ & $1,..\, ,(m-1)$ & $m,..\, ,(2m-1)$ & $m\ge3$  \\
Dihedral group $D_{2m}$ & $2m$ ($m$ even) & $m$ & $1,..\, ,({m\over2}-1)$ & ${m\over2},..\, ,(m-1)$ & $m\ge4$ \\
Rotation group $R_m$ & $m$ ($m$ odd) & $m$ & $m-1$ & $\{\}$ & $m\ge3$ \\
Rotation group $R_m$ & $m$ ($m$ even) & ${m\over2}$ & $m-1$ & $\{\}$ & $m\ge4$ \\
\hline
\end{tabular} 
\end{center}
\end{table}

We now list some additional calculations (Table \ref{selected2orbitdesigns}).
These include the Heisenberg group $\cH_d$ in $d$ dimensions,
which is generated by a cyclic shift $S$ and the modulation $\gO$, where
$Se_j=e_{j+1}$, $\gO e_j = \go^j e_j$, $j\in\ZZ_d$.

\setlength{\tabcolsep}{3pt}
\renewcommand{\arraystretch}{1.2}
\begin{table}[H]
\caption{The unions of pairs of orbits for selected 
groups.
Here $G=G(m,p,n)$, $p\mid m$   ($\vert G\vert=m^n n!/p$), 
is the infinite family in Shephard-Todd classification,
and the groups $B_d$, $D_d$, $H_3$, $F_4$ in brackets are from the Coxeter classification
of real reflection groups.  }
\begin{center}
\label{selected2orbitdesigns}
\begin{tabular}{|p{3.9cm}| p{0.5cm}|p{1.3cm}|p{1.3cm}||p{1.1cm}|p{1.1cm}|p{4cm}|}
\hline
Group & $d$ & order & {\#}lines& $t_{\rm generic}$ & $t_{\rm pairs}$ & comments \\ \hline
$G_{23}$ (ST 23, $H_3$) & 3 & 120 & 60 & 1-2 & 3-4 & real group \\
$G_{24}$ (ST 24) & 3 & 336 & 168 & 1-2 & 3 & complex group \\
$G_{25}$ (ST 25) & 3 & 648 & 216 & 1-2 & 3 & complex group \\
$G_{26}$ (ST 26) & 3 & 1296 & 216  & 1-2 & 3 & complex group \\
$G_{28}$ (ST 28, $F_4$) & 4 & 1152 & 576 & 1-2 & 3 & real group \\
$G(2,2,4)$ ($D_4$) & 4 & 192 & 96 & 1 & $\{\}$ & Example \ref{exceptions} \\
$G(2,2,d)$ ($D_d$) & $d$ & $2^{d-1}d!$ & & 1 & 2 & $d\ne 4$, $3\le d\le 7$ \\
$G(2,1,4)$ ($B_4$) & 2 & 348 & 192 & 1 & 2-3 & \\
$G(2,1,d)$ ($B_d$) & $d$ & $2^{d}d!$ & $2^{d-1}d!$ & 1 & 2 & $d\ne 2$, $2\le d\le 6$ \\
$\cH_2=D_8$ & 2 & 8 & 4 & 1 & 2-3 & $G(2,1,2)\cong G(4,4,2)$ \\
$\cH_d$ & $d$ & $d^3$ & $d^2$ & 1 & $\{\}$& $d\ge 3$, Example \ref{exceptions} \\
\hline
\end{tabular} 
\end{center}
\end{table}

\begin{example}
\label{exceptions}
The real reflection group $G(2,2,4)$ was the only complex reflection 
group we considered for which pairs of orbits do not give 
spherical $(t,t)$-designs. Even in this case, some pairs of highly symmetric
tight frames still give higher order $(t,t)$-designs
(see Table \ref{CoxetDd}).
Also, the Heisenberg
groups $\cH_d$, $d\ge3$ (which are not complex reflection groups) do 
not have the property that pairs of orbits give 
spherical $(t,t)$-designs. 
\end{example}


%
%
%


\section{Summary of calculations}
\label{summarysect}

In the following tables we summarise our calculations to find a 
weighting $(\gb_X,\gb_Y)$ so that a union of highly symmetric tight frames  
$X$ and $Y$ for a 
complex reflection group $G$ is a spherical $(t,t)$-design
with $t$ larger than the generic strength of an orbit. 

Here $ST$ is the Shephard-Todd number of $G$ acting on $\Rd,\Cd$,
$t$ is the strength of the union $X\cup Y$ above that of a generic orbit,
and $n$ is the number of lines in the union.  
We note that there is at least one highly symmetric tight frame
for each group. Such frames are identified by their number of lines,
with $4_2$ in Table \ref{2orbitdesigns2dims}
indicating that either of the two highly symmetric tight 
frames of $4$ points/lines can be taken.

\subsection{Primitive complex reflection groups}

\setlength{\tabcolsep}{3pt}
\renewcommand{\arraystretch}{1.2}
\begin{table}[H]
\caption{The unions of pairs of the highly symmetric tight frames 
for $\CC^2$ given by the complex reflection groups with Shephard-Todd
numbers 4-22
 which are 
$(t,t)$-designs. 
The groups $ST = 4,5,8,12,16,20,22$ have
only one highly symmetric tight frame.  }
\begin{center}
\label{2orbitdesigns2dims}
\begin{tabular}{|p{0.9cm}| p{0.7cm}|p{1.8cm}||p{0.8cm}|p{1.2cm}|p{2cm}|p{2cm}|p{2cm}|}
\hline
ST & $d$ & order & $t$ & $n$ & $|X|,\ |Y|$ & $\gb_X,\ \gb_Y$ & $\hat w_X,\ \hat w_Y$ \\ \hline
6 & 2 & 48 & 3 & 6 & 4,\ 6 & $0,\ 1$ & $0,\ 1$  \\ 
7 & 2 & 144 & 3 & 8 & 4,\ 4 & ${1\over2},\ {1\over2}$ & $1,\ 1$  \\ 
 &&& 3 & 6 & $4_2$,\ $6$ & $0,\ 1$ & $0,\ 1$ \\ 
9 & 2 & 192 & 4-5 & 18 & 6,\ 12 & ${1\over5},\ {4\over5}$ & ${3\over5},\ {6\over5}$  \\ 
10 & 2 & 288 & 4-5 & 14 & 6,\ 8 & ${2\over5},\ {3\over5}$ & ${14\over15},\ {21\over20}$  \\ 
11 & 2 & 576 & 4-5 & 14 & 6,\ 8 & ${2\over5},\ {3\over5}$ & ${14\over15},\ {21\over20}$  \\ 
 &&& 4-5 & 18 & 6,\ 12 & ${1\over5},\ {4\over5}$ & ${3\over5},\ {6\over5}$ \\ 
 &&& 4-5 & 20 & 8,\ 12 & $-{3\over5},\ {8\over5}$ & $-{3\over2},\ {8\over3}$ \\ 
13 & 2 & 96 & 4-5 & 18 & 6,\ 12 & ${1\over5},\ {4\over5}$ & ${3\over5},\ {6\over5}$  \\ 
14 & 2 & 144 & 4-5 & 20 & 8,\ 12 & $-{3\over5},\ {8\over5}$ & $-{3\over2},\ {8\over3}$ \\ 
15 & 2 & 288 & 4-5 & 14 & 6,\ 8 & ${2\over5},\ {3\over5}$ & ${14\over15},\ {21\over20}$  \\ 
 &&& 4-5 & 18 & 6,\ 12 & ${1\over5},\ {4\over5}$ & ${3\over5},\ {6\over5}$ \\ 
 &&& 4-5 & 20 & 8,\ 12 & $-{3\over5},\ {8\over5}$ & $-{3\over2},\ {8\over3}$ \\ 
17 & 2 & 1200 & 6-9 & 42 & 12,\ 30 & ${5\over21},\ {16\over21}$ & ${5\over6},\ {16\over15}$  \\ 
18 & 2 & 1800 & 6-9 & 32 & 12,\ 20 & ${5\over 14},\ {9\over 14}$ & ${20\over21},\ {36\over35}$  \\ 
19 & 2 & 3600 & 6-9 & 32 & 12,\ 20 & ${5\over 14},\ {9\over 14}$ & ${20\over21},\ {36\over35}$  \\ 
 &&& 6-9 & 42 & 12,\ 30 & ${5\over21},\ {16\over21}$ & ${5\over6},\ {16\over15}$ \\ 
 &&& 6-9 & 50 & 20,\ 30 & $-{9\over7},\ {16\over7}$ & $-{45\over14},\ {80\over21}$ \\ 
21 & 2 & 720 & 6-9 & 50 & 20,\ 30 & $-{9\over7},\ {16\over7}$ & $-{45\over14},\ {80\over21}$  \\  
\hline
\end{tabular} 
\end{center}
\end{table}

\setlength{\tabcolsep}{3pt}
\renewcommand{\arraystretch}{1.1}
\begin{table}[H]
\caption{The unions of pairs of the highly symmetric tight frames
for $\CC^d$ given by the complex reflection groups with Shephard-Todd
numbers in the range 23-37
 which are 
$(t,t)$-designs.
The groups $ST = 23,28,30,35,36,37$ are real reflection groups.  }
\begin{center}
\label{2orbitcomplexdesignsdims3-6}
\begin{tabular}{|p{0.9cm}| p{0.7cm}|p{1.8cm}||p{0.8cm}|p{1.2cm}|p{2.1cm}|p{2cm}|p{2cm}|}
\hline
ST & $d$ & order & $t$ & $n$ & $|X|,\ |Y|$ & $\gb_X,\ \gb_Y$ & $\hat w_X,\ \hat w_Y$ \\ \hline
24 & 3 & 336 & 3 & 49 & 21,\ 28 & ${26\over35},\ {9\over35}$ & ${26\over15},\ {9\over20}$  \\ 
25 & 3 & 648 & 3 & 21 & 9,\ 12 & ${2\over 5},\ {3\over 5}$ & ${14\over15},\ {21\over20}$  \\ 
26 & 3 & 1296 & 3 & 21 & 9,\ 12 & ${2\over 5},\ {3\over 5}$ & ${14\over15},\ {21\over20}$  \\ 
 &&& 3 & 45 & 9,\ 36 & $-{1\over5},\ {6\over5}$ & $-1,\ {3\over2}$ \\ 
 &&& 3-4 & 48 & 12,\ 36 & ${1\over5},\ {4\over5}$ & ${4\over5},\ {16\over15}$ \\ 
27 & 3 & 2160 & 4 & 81 & 36,\ 45 & ${5\over 9},\ {4\over 9}$ & ${5\over4},\ {4\over5}$  \\ 
 &&& 4 & 96 & 36,\ $60_2$ & ${5\over8},\ {3\over8}$ & ${5\over 3},\ {3\over5}$ \\ 
 &&& 4 & 105 & 45,\ $60_2$ & $4,\ -3$ & ${28\over 3},\ -{21\over 4}$ \\ 

\hline

29 & 4 & 7680 & 3 & 60 & 20,\ 40 & ${1\over 3},\ {2\over 3}$ & $1,\ 1$  \\ 
 &&& 3 & 100 & 20,\ $80_2$ & $-{1\over 3},\ {4\over 3}$ & $-{5\over 3},\ {5\over 3}$ \\ 
 &&& 3 & 120 & 40,\ $80_2$ & ${1\over 3},\ {2\over 3}$ & $1,1$ \\ 
 &&& 3 & 180 & 20,\ 160 & $-{7\over 11},\ {18\over 11}$ & $-{63\over 11},\ {81\over 44}$ \\ 
 &&& 3 & 200 & 40,\ 160 & ${7\over 16},\ {9\over 16}$ & ${35\over 16},\ {45\over 64}$ \\ 
 &&& 3 & 240 & $80_2$,\ 160 & ${14\over 5},\ -{9\over 5}$ & ${42\over 5},\ -{27\over 10}$ \\ 

31 & 4 & 46080 & 4-5 & 540 & 60,\ 480 & ${5\over 21},\ {16\over 21}$ & ${15\over7},\ {6\over7}$  \\ 
 &&& 4-5 & 1020 & 60,\ 960 & ${1\over 28},\ {27\over 28}$ & ${17\over28},\ {459\over 448}$ \\ 
&&& 4-5 & 1440 & 480,\ 960 & $-{16\over119},\ {135\over119}$ & $-{48\over119},\ {405\over 238}$ \\ 
32 & 4 & 155520 & 4-5 & 400 & 40,\ 360 & ${1\over 7},\ {6\over 7}$ & ${10\over7},\ {20\over21}$  \\ 
\hline
33 & 5 & 51840 & 3 & 85 & 40,\ 45 & ${3\over 7},\ {4\over 7}$ & ${51\over 56},\ {68\over 63}$  \\ 
 &&& 3 & 256 & 40,\ 216 & ${3\over 28},\ {25\over 28}$ & ${24\over 35},\ {200\over 189}$ \\ 
 &&& 3 & 261 & 45,\ 216 & $-{4\over21},\ {25\over 21}$ & $-{116\over 105},\ {725\over 504}$ \\ 
&&& 3 & 580 & 40,\ 540 & $-{15\over 49},\ {64\over 49}$ & ${-{435\over 98}},\ {1856\over 1323}$ \\ 
 & & & 3 & 585 & 45,\ 540 & ${5\over21},\ {16\over 21}$ & ${65\over 21},\ {52\over 63}$ \\ 
 & & & 3 & 756 & 216,\ 540 & ${125\over189},{64\over 189}$ &
${125\over 54},\ {64\over 135}$ \\ 
\hline
34 & 6 & 39191040 & 4 & 672 & 672,\ $|Y|$ & $1,\ 0$ & $1,\ 0$  \\ 
 &&& 4 & 9072 & 9072,\ $|Y|$ & $1,\ 0$ & $1,\ 0$ \\ 
 &&& 4-5 & 3528 & 126,\ 3402 & ${1\over9},\ {8\over9}$ & ${28\over9},\ {224\over243}$ \\ 
 &&& 4 & 5166 & 126,\ 5040 & $-{1\over4},\ {5\over4}$ & $-{41\over4},\ {41\over32}$ \\ 
 &&& 4 & 8442 & 3402,\ 5040 & ${8\over13},\ {5\over13}$ & ${536\over351},\, {67\over104}$ \\ 
 &&& 4-5 & $9744_2$ & 672,\ $9072_2$ & $-{1\over5},\ {6\over5}$ & $-{29\over10},\ {58\over45}$ \\ 
 &&& 4 & 18144 & 9072,\ 9072 & $\gb,\ 1-\gb$ & $2\gb,2(1-\gb)$ \\ 
 &&& $\ge4$ & 27342 & 126,\ 27216 & $-{13\over612},\ {625\over612}$ & $-{2821\over612},\ {135625\over132192}$ \\ 
 &&& $\ge4$ & 30618 & 3402,\ 27216 & ${104\over729},\ {625\over729}$ & ${104\over81},\ {625\over648}$ \\ 
 &&& $\ge4$ & 32256 & 5040,\ 27216 & $-{13\over112},\ {125\over112}$ & $-{26\over35},\ {250\over189}$ \\ 
 &&& $\ge4$ & 45486 & 126,\ 45360 & $-{3\over37},\ {40\over37}$ & $-{1083\over37},\ {361\over333}$ \\ 

\hline
\end{tabular} 
\end{center}
\end{table}

\setlength{\tabcolsep}{3pt}
\renewcommand{\arraystretch}{1.1}
 \begin{center}
\LTcapwidth=\linewidth
\begin{longtable}{|p{0.9cm}|p{0.9cm}| p{0.7cm}|p{1.8cm}||p{0.8cm}|p{1.2cm}|p{2.3cm}|p{2.6cm}|p{2.6cm}|}
 \caption{ 
The unions of pairs of the highly symmetric tight frames
for $\Rd$ given by the real reflection groups $G$ with Shephard-Todd
numbers $23,28,30,35,36,37$ which are spherical 
$(t,t)$-designs. The Coxeter classification names are included
under $G$.  }
\label{2orbitRealdesignsdims} \\
\hline
ST & $G$ & $d$ & order & $t$ & $n$ & $|X|,\ |Y|$ & $\gb_X,\ \gb_Y$ & $\hat w_X,\ \hat w_Y$ \\ \hline
23 & $H_3$ & 3 & 120 & 3-4 & 16 & 6,\ 10 & ${5\over14},\ {9\over14}$ & ${20\over21},\ {36\over35}$  \\
 &&&& 3-4 & 21 & 6,\ 15 & ${5\over21},\ {16\over21}$ & ${5\over6},\ {16\over15}$ \\
 &&&& 3-4 & 25 & 10,\ 15 & $-{9\over7},\ {16\over7}$ & $-{45\over14},\ {80\over21}$ \\
\hline
28 & $F_4$ & 4 & 1152 & 3 & 24 & 12,\ 12 & ${1\over2},\ {1\over2}$ & $1,\ 1$  \\
 &&&& 3 & $60_2$ & 12,\ 48 & $-{1\over8},\ {9\over8}$ & $-{5\over8},\ {45\over32}$ \\
 &&&& 3 & $60_2$ & 12,\ 48 & ${1\over10},\ {9\over10}$ & ${1\over2},\ {9\over8}$ \\
 &&&& 3 & 96 & 48,\ 48 & ${1\over2},\ {1\over2}$ & $1,\ 1$ \\

30 & $H_4$ & 4 & 14400 & 6-9 & 360 & 60,\ 300 & ${5\over21},\ {16\over21}$ & ${10\over7},\ {32\over35}$  \\
 &&&& 6-9 & 420 & 60,\ 360 & ${3\over28},\ {25\over28}$ & ${3\over4},\ {25\over24}$ \\
 &&&& 6-9 & 660 & 60,\ 600 & ${65\over308},\ {234\over308}$ & ${65\over28},\ {243\over280}$ \\
 &&&& 6-9 & 660 & 300,\ 360 & $-{48\over77},\ {125\over77}$ & $-{48\over35},\ {125\over42}$ \\
 &&&& 6-9 & 900 & 300,\ 600 & $-{208\over35},\ {243\over35}$ & $-{624\over35},\ {729\over70}$ \\
 &&&& 6-9 & 960 & 360,\ 600 & ${1625\over896},\ -{729\over896}$ & ${1625\over336},\ -{729\over560}$ \\
\hline
35 & $E_6$ & 6 & 51840 & 3 & 63 & 27,\ 36 & ${2\over5},\ {3\over5}$ & ${14\over15},\ {21\over20}$  \\
 &&&& 3 & 243 & 27,\ 216 & ${2\over27},\ {25\over27}$ & ${2\over3},\ {25\over24}$ \\
 &&&& 3 & 252 & 36,\ 216 & $-{3\over22},\ {25\over22}$ & $-{21\over22},\ {175\over132}$ \\
 &&&& 3 & 387 & 27,\ 360 & ${2\over11},\ {9\over11}$ & ${86\over33},\ {387\over440}$ \\
 &&&& 3 & 396 & 36,\ 360 & $-{1\over2},\ {3\over2}$ & $-{11\over2},\ {33\over20}$ \\
 &&&& 3 & 576 & 216,\ 360 & ${25\over16},\ -{9\over16}$ & ${25\over6},\ -{9\over10}$ \\
\hline
36 & $E_7$ & 7 & 2903040 & 3 & 91 & 28,\ 63 & ${3\over11},\ {8\over11}$ & ${39\over44},\ {104\over99}$  \\
 &&&& 3 & 316 & 28,\ 288 & ${6\over55},\ {49\over55}$ & ${474\over385},\ {3871\over3960}$ \\
 &&&& 3 & 351 & 63,\ 288 & $-{16\over33},\ {49\over33}$ & $-{208\over77},\ {637\over352}$ \\
 &&&& 3 & 406 & 28,\ 378 & $-{9\over55},\ {64\over55}$ & $-{261\over110},\ {1856\over1485}$ \\
 &&&& 3 & 441 & 63,\ 378 & ${3\over11},\ {8\over11}$ & ${21\over11},\ {28\over33}$ \\
 &&&& 3 & 666 & 288,\ 378 & ${147\over275},\ {128\over275}$ & ${5439\over4400},\ {4736\over5775}$ \\
 &&&& 3 & 1036 & 28,\ 1008 & ${7\over55},\ {48\over55}$ & ${259\over55},\ {148\over165}$ \\
 &&&& 3 & 1071 & 63,\ 1008 & $-{7\over11},\ {18\over11}$ & $-{119\over11},\ {153\over88}$ \\
 &&&& 3 & 1296 & 288,\ 1008 & ${343\over55},\ -{288\over55}$ & ${3087\over110},\ -{2592\over385}$ \\
 &&&& 3 & 1386 & 378,\ 1008 & ${28\over55},\ {27\over55}$ & ${28\over15},\ {27\over40}$ \\
 &&&& 3 & 2044 & 28,\ 2016 & ${2\over77},\ {75\over77}$ & ${146\over77},\ {1825\over1848}$ \\
 &&&& 3 & 2079 & 63,\ 2016 & $-{16\over209},\ {225\over209}$ & $-{48\over19},\ {675\over608}$ \\
 &&&& 3 & 2304 & 288,\ 2016 & $-{49\over176},\ {225\over176}$ & $-{49\over22},\ {225\over154}$ \\
 &&&& 3 & 2394 & 378,\ 2016 & ${128\over803},\ {675\over803}$ & ${2432\over2409},\ {12825\over12848}$ \\
 &&&& 3 & 3024 & 1008,\ 2016 & $-{32\over143},\ {175\over143}$ & $-{96\over143},\ {525\over286}$ \\
 &&&& 3 & 5068 & 28,\ 5040 & ${17\over209},\ {192\over209}$ & ${3077\over209},\ {2896\over3135}$ \\
 &&&& 3 & 5103 & 63,\ 5040 & $-{17\over55},\ {72\over55}$ & $-{1377\over55},\ {729\over550}$ \\
 &&&& 3 & 5328 & 288,\ 5040 & $-{883\over319},\ {1152\over319}$ & $-{30821\over638},\ {42624\over11165}$ \\
 &&&& 3 & 5418 & 378,\ 5040 & ${17\over44},\ {27\over44}$ & ${731\over132},\ {1161\over1760}$ \\
 &&&& 3 & 6048 & 1008,\ 5040 & $-{17\over11},\ {28\over11}$ & $-{102\over11},\ {168\over55}$ \\
 &&&& 3 & 7056 & 2016,\ 5040 & ${425\over297},\ -{128\over297}$ & ${2975\over594},\ -{896\over1485}$ \\
\hline
37 & $E_8$ & 8 & 696729600 & 4-5 & 1200 & 120,\ 1080 & ${1\over7},\ {6\over7}$ & ${10\over7},\ {20\over21}$  \\
 &&&& 4-5 & 3480 & 120,\ 3360 & $-{1\over8},\ {9\over8}$ & $-{29\over8},\ {261\over224}$ \\
 &&&& 4-5 & 4440 & 1080,\ 3360 & ${2\over5},\ {3\over5}$ & ${74\over45},\ {111\over140}$ \\
 &&&& 4-5 & 8760 & 120,\ 8640 & ${3\over35},\ {32\over35}$ & ${219\over35},\ {292\over315}$ \\
 &&&& 4-5 & 9720 & 1080,\ 8640 & $-{9\over7},\ {16\over7}$ & $-{81\over7},\ {18\over7}$ \\
 &&&& 4-5 & 12000 & 3360,\ 8640 & ${27\over59},\ {32\over59}$ & ${675\over413},\ {400\over531}$ \\
 &&&& 4-5 & 30360 & 120,\ 30240 & ${1\over55},\ {54\over55}$ & ${23\over5},\ {69\over70}$ \\
 &&&& 4-5 & 31320 & 1080,\ 30240 & $-{1\over8},\ {9\over8}$ & ${29\over8},\ {261\over224}$ \\
 &&&& 4-5 & 33600 & 3360,\ 30240 & ${1\over7},\ {6\over7}$ & ${10\over7},\ {20\over21}$ \\
 &&&& 4-5 & 34680 & 120,\ 34560 & ${33\over376},\ {343\over376}$ & ${9537\over376},\ {99127\over108288}$ \\
 &&&& 4-5 & 35640 & 1080,\ 34560 & $-{198\over145},\ {343\over145}$ & $-{6534\over145},\ {11319\over4640}$ \\
 &&&& 4-5 & 37920 & 3360,\ 34560 & ${297\over640},\ {343\over640}$ & ${23463\over4480},\ {27097\over46080}$ \\
 &&&& 4-5 & 38880 & 8640,\ 30240 & $-{16\over65},\ {81\over65}$ & $-{72\over65},\ {729\over455}$ \\
 &&&& 4-5 & 43200 & 8640,\ 34560 & ${352\over9},\ -{343\over9}$ & ${1760\over9},\ -{1715\over36}$ \\
 &&&& 4-5? & 64800 & 30240,\ 34560 & ${1782\over1439},\ -{343\over1439}$ & ${26730\over10073},\ -{5145\over11512}$ \\
 &&&& 4-5? & 121080 & 120,\ 120960 & ${3\over53},\ {50\over53}$ & ${3027\over53},\ {25225\over26712}$ \\
 &&&& 4-5? & 122040 & 1080,\ 120960 & $-{9\over16},\ {25\over16}$ & $-{1017\over16},\ {2825\over1792}$ \\
 &&&& 4-5? & 124320 & 3360,\ 120960 & ${27\over77},\ {50\over77}$ & ${999\over77},\ {925\over1386}$ \\

\hline
\end{longtable} 
\end{center}

\subsection{Imprimitive complex reflection groups}

\setlength{\tabcolsep}{3pt}
\renewcommand{\arraystretch}{1.15}
\begin{table}[H]
\caption{Selected examples for 
the Coxeter groups $A_{d}=G(1,1,d+1)\cong S_{d+1}$, $d\ge2$.
These $n$ vector spherical $(2,2)$-designs give rise to spherical $5$-designs
with $2n$ vectors.
}
\begin{center}
\label{CoxeterAd}
\begin{tabular}{|p{1.8cm}| p{0.7cm}|p{2.1cm}||p{0.8cm}|p{1.2cm}|p{2.0cm}|p{2.0cm}|p{2.0cm}|}
\hline
ST & $d$ & order & $t$ & $n$ & $|X|,\ |Y|$ & $\gb_X,\ \gb_Y$ & $\hat w_X,\ \hat w_Y$ \\ \hline
(1,1,4) & 3 & 24 & 2 & 7 & 3,\ 4 & ${2\over5},\ {3\over5}$ & ${14\over15},\ {21\over20}$  \\
(1,1,5) & 4 & 120 & 2 & 15 & 5,\ 10 & ${2\over5},\ {3\over5}$ & ${6\over5},\ {9\over10}$  \\
(1,1,6) & 5 & 720 & 2 & 16 & 6,\ 10 & ${5\over14},\ {9\over14}$ & ${20\over21},\ {36\over35}$  \\
 &&& 2 & 21 & 6,\ 15 & ${5\over21},\ {16\over21}$ & ${5\over6},\ {16\over15}$ \\
(1,1,7) & 6 & 5040 & 2 & 28 & 7,\ 21 & ${3\over28},\ {25\over28}$ & ${3\over7},\ {25\over21}$  \\
 &&& 2 & 42 & 7,\ 35 & ${2\over7},\ {5\over7}$ & ${12\over7},\ {6\over7}$ \\
(1,1,8) & 7 & 40320 & 2 & 28 & 28,\ $|Y|$ & ${1},\ {0}$ & ${1},\ {0}$  \\
 &&& 2 & 43 & 8,\ 35 & ${7\over27},\ {20\over27}$ & ${301\over216},\ {172\over189}$ \\
 &&& 2 & 64 & 8,\ 56 & ${7\over32},\ {250\over32}$ & ${7\over4},\ {25\over28}$ \\
(1,1,9) & 8 & 362880 & 2 & 45 & 9,\ 36 & $-{4\over45},\ {49\over45}$ & $-{4\over9},\ {49\over36}$  \\
 &&& 2 & 93 & 9,\ 84 & ${4\over25},\ {21\over25}$ & ${124\over75},\ {93\over100}$ \\
(1,1,10) & 9 & 3628800 & 2 & 55 & 10,\ 45 & $-{9\over55},\ {64\over55}$ & $-{9\over10},\ {64\over45}$  \\
 &&& 2 & 130 & 10,\ 120 & ${6\over55},\ {49\over55}$ & ${78\over55},\ {637\over660}$ \\
(1,1,11) & 10 & 39916800 & 2 & 66 & 11,\ 55 & $-{5\over22},\ {27\over22}$ & $-{15\over11},\ {81\over55}$  \\
 &&& 2 & 176 & 11,\ 165 & ${5\over77},\ {72\over77}$ & ${80\over77},\ {384\over385}$ \\
\hline
\end{tabular}
\end{center}
\end{table}

\setlength{\tabcolsep}{3pt}
\renewcommand{\arraystretch}{1.15}
\begin{table}[H]
\caption{Selected examples for 
the Coxeter groups $B_{d}=G(2,1,d)$, $d\ge2$,
for orbits of the vectors $x=e_1+e_2+\cdots+e_k$, $1\le k\le d$.
}
\begin{center}
\label{CoxeterBd}
\begin{tabular}{|p{1.8cm}| p{0.7cm}|p{2.1cm}||p{0.8cm}|p{1.2cm}|p{2.0cm}|p{2.0cm}|p{2.0cm}|}
\hline
ST & $d$ & order & $t$ & $n$ & $|X|,\ |Y|$ & $\gb_X,\ \gb_Y$ & $\hat w_X,\ \hat w_Y$ \\ \hline
(2,1,2) & 2 & 8 & 2-3 & 4 & 2,\ 2 & ${1\over2},\ {1\over2}$ & ${1},\ {1}$  \\
(2,1,3) & 3 & 48 & 2 & 7 & 3,\ 4 & ${2\over5},\ {3\over5}$ & ${14\over15},\ {21\over20}$  \\
 &&& 2 & 9 & 3,\ 6 & ${1\over5},\ {4\over5}$ & ${3\over5},\ {6\over5}$ \\
(2,1,4) & 4 & 384 & 2 & 12 & 4,\ 8 & ${1\over3},\ {2\over3}$ & ${1},\ {1}$  \\
 &&& 2 & 12 & 12,\ $|Y|$ & ${1},\ {0}$ & ${1},\ {0}$ \\
 &&& 2 & 20 & 4,\ 16 & ${1\over4},\ {3\over4}$ & ${5\over4},\ {15\over16}$ \\
(2,1,5) & 5 & 3840 & 2 & 21 & 5,\ 16 & ${2\over7},\ {5\over7}$ & ${6\over5},\ {15\over16}$  \\
 &&& 2-3 & 45 & 5,\ 40 & ${1\over7},\ {6\over7}$ & ${9\over7},\ {27\over28}$ \\
(2,1,6) & 6 & 46080 & 2 & 36 & 6,\ 30 & $-{1\over4},\ {5\over4}$ & $-{3\over2},\ {3\over2}$  \\
 &&& 2 & 38 & 6,\ 32 & ${1\over4},\ {3\over4}$ & ${19\over12},\ {57\over64}$ \\
(2,1,7) & 7 & 645120 & 2 & 49 & 7,\ 42 & $-{1\over3},\ {4\over3}$ & $-{7\over3},\ {14\over9}$  \\
 &&& 2 & 71 & 7,\ 64 & ${2\over9},\ {7\over9}$ & ${142\over63},\ {497\over576}$ \\
(2,1,8) & 8 & 10321920 & 2 & 64 & 8,\ 56 & $-{2\over5},\ {7\over5}$ & $-{16\over5},\ {8\over5}$  \\
 &&& 2 & 136 & 8,\ 128 & ${1\over5},\ {4\over5}$ & ${17\over5},\ {17\over20}$ \\
 &&& 2-3 & 184 & 56,\ 128 & ${7\over15},\ {8\over15}$ & ${23\over15},\ {23\over30}$ \\
 &&& 2-3 & 568 & 8,\ 560 & ${1\over15},\ {14\over15}$ & ${71\over15},\ {71\over75}$ \\
(2,1,9) & 9 & 185794560 & 2 & 81 & 9,\ 72 & $-{5\over11},\ {16\over11}$ & $-{45\over11},\ {18\over11}$  \\
 &&& 2 & 265 & 9,\ 256 & ${2\over11},\ {9\over11}$ & ${530\over99},\ {2385\over2816}$ \\
(2,1,10) & 10 & 3715891200 & 2 & 100 & 10,\ 90 & $-{1\over2},\ {3\over2}$ & $-{5},\ {5\over3}$  \\
 &&& 2 & 490 & 10,\ 480 & $-{1\over8},\ {9\over8}$ & $-{49\over48},\ {147\over128}$ \\
 &&& 2 & 522 & 10,\ 512 & ${1\over6},\ {5\over6}$ & ${87\over10},\ {435\over512}$ \\
 &&& 2-3 & 7770 & 90,\ 7680 & ${3\over10},\ {7\over10}$ & ${259\over10},\ {1813\over2560}$ \\
\hline
\end{tabular}
\end{center}
\end{table}

\setlength{\tabcolsep}{3pt}
\renewcommand{\arraystretch}{1.15}
\begin{table}[H]
\caption{Selected examples for the Coxeter groups 
$D_{d}={\rm G}(2,2,d)$, $d\ge3$, for 
orbits of the vectors $x=e_1+e_2+\cdots+e_k$, $1\le k\le d$.
Note that $G(2,2,2)$ is not irreducible, and
$G(2,2,3)\cong G(1,1,4)$. }
\begin{center}
\label{CoxetDd}
\begin{tabular}{|p{1.9cm}| p{0.7cm}|p{2.1cm}||p{0.8cm}|p{1.2cm}|p{2.1cm}|p{2.1cm}|p{2.1cm}|}
\hline
ST & $d$ & order & $t$ & $n$ & $|X|,\ |Y|$ & $\gb_X,\ \gb_Y$ & $\hat w_X,\ \hat w_Y$ \\ \hline
(2,2,3) & 3 & 24 & 2 & 7 & 3,\ 4 & ${2\over5},\ {3\over5}$ & ${14\over15},\ {21\over20}$  \\
 &&& 2 & 9 & 3,\ 6 & ${1\over5},\ {4\over5}$ & ${3\over5},\ {6\over5}$ \\
 &&& 2 & 10 & 4,\ 6 & $-{3\over5},\ {8\over5}$ & $-{3\over2},\ {8\over3}$ \\
(2,2,4) & 4 & 192 & \{\} & 8 & 4,\ 4 & & \\
 &&& 2 & 12 & 12,\ $|Y|$ & ${1},\ {0}$ & ${1},\ {0}$ \\
 &&& 2 & 20 & 4,\ 16 & ${1\over4},\ {3\over4}$ & ${5\over4},\ {15\over16}$ \\
 &&& $\{\}$ & 20 & 4,\ 16 & & \\
(2,2,5) & 5 & 1920 & 2 & 21 & 5,\ 16 & ${2\over7},\ {5\over7}$ & ${6\over5},\ {15\over16}$  \\
(2,2,6) & 6 & 23040 & 2 & 22 & 6,\ 16 & ${1\over4},\ {3\over4}$ & ${11\over12},\ {33\over32}$  \\
(2,2,7) & 7 & 322560 & 2 & 49 & 7,\ 42 & $-{1\over3},\ {4\over3}$ & $-{7\over3},\ {14\over9}$  \\
 &&& 2 & 71 & 7,\ 64 & ${2\over9},\ {7\over9}$ & ${142\over63},\ {497\over576}$ \\
 &&& 2 & 140 & 140,\ $|Y|$ & $1,\ 0$ & $1,\ 0$ \\

(2,2,8) & 8 & 5160960 & 2 & 64 & 8,\ 56 & $-{2\over5},\ {7\over5}$ & $-{16\over5},\ {8\over5}$  \\
 &&& 2 & 72 & 8,\ 64 & ${1\over5},\ {4\over5}$ & ${9\over5},\ {9\over10}$ \\
 &&& 2-3 & 120 & 56,\ 64 & ${7\over15},\ {8\over15}$ & ${1},\ {1}$ \\
 &&& 2-3 & 568 & 8,\ 560 & ${1\over15},\ {14\over15}$ & ${71\over15},\ {71\over75}$ \\
(2,2,9) & 9 & 92897280 & 2 & 81 & 9,\ 72 & $-{5\over11},\ {16\over11}$ & $-{45\over11},\ {18\over11}$  \\
 &&& 2 & 265 & 9,\ 256 & ${2\over11},\ {9\over11}$ & ${530\over99},\ {2385\over2816}$ \\
(2,1,10) & 10 & 1857945600 & 2 & 100 & 10,\ 90 & $-{1\over2},\ {3\over2}$ & $-{5},\ {5\over3}$  \\
 &&& 2 & 266 & 10,\ 256 & ${1\over6},\ {5\over6}$ & ${133\over30},\ {665\over768}$ \\
 &&& 2 & 1680 & 1680,\ $|Y|$ & ${1},\ {0}$ & ${1},\ {0}$ \\
 &&& 2-3 & 7770 & 90,\ 7680 & ${3\over10},\ {7\over10}$ & ${259\over10},\ {1813\over2560}$ \\
\hline
\end{tabular}
\end{center}
\end{table}

\vfil\eject
\subsection{Observations and examples}

We first observe that it is possible to have a weighting $(\gb_X,\gb_Y)$
with a {\em negative value} that gives a union $X\cup Y$
which is a $(t,t)$-design of higher order. We will refer to a $(t,t)$-design
with some negative weights as a {\bf signed} $(t,t)$-design.
Signed $(1,1)$-designs were first studied in \cite{PW02}, where
they were called {\it signed tight frames} and defined as systems
with
$$ f=\sum_j c_j\inpro{f,\phi_j}\phi_j, \qquad \forall f\in\Fd, $$
where $c_j\in\RR$ and $\phi_j\in\SS$. The equivalence of these
two notions is easily proved.

In some cases, pairs of orbits can give $(t,t)$-designs with strength
$t_{\rm pairs}>t_{\rm generic}+1$.
We illustrate the mechanism
for this with an example. Consider the Shephard-Todd groups numbered 9 to 15.
For these, an orbit gives a cubature rule for $P=\Hom(t,t)$, $t=1,2,3,5$, as does
any union of orbits. Thus a union of orbits is a $(4,4)$-design if and
only if it is a $(5,5)$-design. There are also examples, such as the
Shephard-Todd group 26, where some, but not all, pairs of orbits have strength
greater than $t_{\rm generic}+1$.

In \cite{HW18} a numerical study was done to find ``putatively optimal''
$(t,t)$-designs. We now consider our constructions in relation
to the table in \cite{HW18}
 (and \cite{W18}). 

\begin{example}
For $\CC^2$ the putatively optimal $(t,t)$-designs for $t=3,4,5$ come as 
highly symmetric tight frames (one orbit). 
For $t=8,9$ the putatively 
optimal number of vectors was estimated to be $37$ and $44$. Since
we have constructed a $(9,9)$-design of $32$ vectors for $\CC^2$
as a union of orbits of size $12$ and $20$,
these numbers can be improved.
\end{example}

\begin{example}
For $\CC^3$ the putatively optimal $(3,3)$-design had $22$ vectors,
and we give one with $21$ vectors.
The putatively optimal $(4,4)$-design had $47$ vectors, and we
give one with $48$ vectors.
\end{example}

\begin{example}
For $\CC^4$ the putatively optimal $(4,4)$-design had more than $85$ vectors,
and there was no estimate for $(5,5)$-designs.
Here we give a $(5,5)$-design with $400$ vectors.
\end{example}

\begin{example}
\label{JosiahEx}
For $\CC^5$ the putatively optimal $(3,3)$-design had more than $100$ vectors.
Here we give a $(3,3)$-design of $85$ vectors for $\CC^5$. 
This design was found by \cite{BGMPV19} by optimizing a potential,
and then presented explicitly (in terms of root vectors of $G_{33}$).
\end{example}

\begin{example}
For $\CC^6$ the highly symmetric tight frame of $672$ vectors
for the group $G_{34}$ was identified as a $(4,4)$-design (higher strength
than a generic orbit), and a pair of orbits gives a $(5,5)$-design
of $3528$ vectors. 
\end{example}

We now consider examples for real reflection groups. 
We note that if $X$ is a spherical $(t,t)$-design of $n$ vectors for $\Rd$,
then $X\cup-X$ (with the same weight on $x$ and $-x$) 
is a spherical $(2t+1)$-design of $2n$ vectors for $\Rd$ (see 
\cite{HW18II}). 

\begin{example}
For the Shephard-Todd group $G_{23}$ a union of orbits of size $6$ and $10$ 
gives a spherical $(4,4)$-design for $\RR^3$ 
(with normalised weights ${20\over21}$, ${36\over35}$).
For the Shephard-Todd group $G(1,1,6)$ acting on five 
dimensional space a union of orbits of size $6$ and $10$ 
gives a spherical $(2,2)$-design for $\RR^5$ 
(with normalised weights ${20\over21}$, ${36\over35}$).
These putatively optimal spherical half-designs were given in 
in \cite{HW18II}. 
\end{example}

\begin{example} 
(Tables \ref{2orbitRealdesignsdims} and \ref{CoxeterBd})
A union of pairs of highly symmetric tight frames for real reflection groups
gives $(3,3)$-designs of
$4$ vectors for $\RR^2$,
$16$ vectors for $\RR^3$,
$24$ vectors for $\RR^4$, 
$45$ vectors for $\RR^5$, 
$63$ vectors for $\RR^6$, 
$91$ vectors for $\RR^7$,  and
$184$ vectors for $\RR^8$.
\end{example}

\begin{example} 
For the Shephard-Todd group $G_{30}$ a generic orbit is a 
$(5,5)$-design for $\RR^4$. A union of highly symmetric
tight frames with $60$ and $300$ vectors gives a $(9,9)$-design 
for $\RR^4$ (with normalised weights ${10\over7}$, ${32\over 35}$).
By taking these vectors and their negatives one obtains a $720$
vector spherical $19$-design for $\RR^4$. It has been shown \cite{BB09}
that there is a single orbit of $G_{30}=W(H_4)$ which gives 
a spherical $19$-design for $\RR^4$. The vectors $x$ giving such
orbits are the roots of the harmonic polynomial of degree $12$ which is
invariant under the action of $G_{30}$, and the orbit size is 
nominally $14400$ vectors.
\end{example}

There is ongoing work of \cite{BGMPV19} on minimising a $p$-frame energy
on a sphere. They present various tables of putatively optimal spherical
$(t,t)$-designs that they have collected from the literature and 
calculated (see Example \ref{JosiahEx}). Many of these a clearly 
examples of our general construction (by a comparision of number
of vectors and weights). These include a $24$-point $(3,3)$-design
for $\RR^4$, a $22$-point $(2,2)$-design for $\RR^6$,
a $63$-point $(3,3)$-design for for $\RR^6$, 
a $91$-point $(3,3)$-design for for $\RR^7$,
and a $21$-point $(3,3)$-design for for $\CC^3$.

\subsection{Conclusion}

We have demonstrated that it is possible to take a union of two orbits to 
obtain a spherical $(t,t)$-design of higher strength than that of 
a generic orbit, i.e., $t_{\rm generic}$, and some of these designs have a minimal 
number of vectors. Given that $t_{\rm generic}\le t_{\rm max}(d)$ 
(for some function $t_{\rm max}$)
for every group acting on $\Rd$, $d\ge3$, it is not possible to find arbitrary
strong $(t,t)$-designs as a single orbit (by selecting a sufficiently 
large group), and so this technique might be useful for finding
designs with strength $t>t_{\rm max}(d)$. 
We note that $t_{\rm max}$ has not yet been determined.

\bibliographystyle{alpha}
\bibliography{union-refs}

\newcommand{\etalchar}[1]{$^{#1}$}
\begin{thebibliography}{BGM{\etalchar{+}}19}

\bibitem[ACFW18]{ACFW18}
Marcus Appleby, Tuan-Yow Chien, Steven Flammia, and Shayne Waldron.
\newblock Constructing exact symmetric informationally complete measurements
  from numerical solutions.
\newblock {\em J. Phys. A}, 51(16):165302, 40, 2018.

\bibitem[Ban79]{B79}
Eiichi Bannai.
\newblock On some spherical {$t$}-designs.
\newblock {\em J. Combin. Theory Ser. A}, 26(2):157--161, 1979.

\bibitem[BB09]{BB09}
Eiichi Bannai and Etsuko Bannai.
\newblock A survey on spherical designs and algebraic combinatorics on spheres.
\newblock {\em European J. Combin.}, 30(6):1392--1425, 2009.

\bibitem[BGM{\etalchar{+}}19]{BGMPV19}
Dmitriy {Bilyk}, Alexey {Glazyrin}, Ryan {Matzke}, Josiah {Park}, and Oleksandr
  {Vlasiuk}.
\newblock {Optimal measures for p-frame energies on spheres}.
\newblock {\em arXiv e-prints}, page arXiv:1908.00885, Aug 2019.

\bibitem[BT06]{BT06}
Christian Bayer and Josef Teichmann.
\newblock The proof of {T}chakaloff's theorem.
\newblock {\em Proc. Amer. Math. Soc.}, 134(10):3035--3040, 2006.

\bibitem[BW13]{BW13}
Helen Broome and Shayne Waldron.
\newblock On the construction of highly symmetric tight frames and complex
  polytopes.
\newblock {\em Linear Algebra Appl.}, 439(12):4135--4151, 2013.

\bibitem[CW17]{CW17}
Tuan-Yow Chien and Shayne Waldron.
\newblock Nice error frames, canonical abstract error groups and the
  construction of {SIC}s.
\newblock {\em Linear Algebra Appl.}, 516:93--117, 2017.

\bibitem[dlHP04]{HP04}
Pierre de~la Harpe and Claude Pache.
\newblock Spherical designs and finite group representations (some results of
  {E}. {B}annai).
\newblock {\em European J. Combin.}, 25(2):213--227, 2004.

\bibitem[GS81]{GS81}
J.-M. Goethals and J.~J. Seidel.
\newblock Cubature formulae, polytopes, and spherical designs.
\newblock In {\em The geometric vein}, pages 203--218. Springer, New
  York-Berlin, 1981.

\bibitem[Hum90]{H90}
James~E. Humphreys.
\newblock {\em Reflection groups and {C}oxeter groups}, volume~29 of {\em
  Cambridge Studies in Advanced Mathematics}.
\newblock Cambridge University Press, Cambridge, 1990.

\bibitem[HW18a]{HW18II}
Daniel Hughes and Shayne Waldron.
\newblock Spherical half designs of high order.
\newblock preprint, 8 2018.

\bibitem[HW18b]{HW18}
Daniel Hughes and Shayne Waldron.
\newblock Spherical $(t,t)$-designs with a small number of vectors.
\newblock preprint, 8 2018.

\bibitem[JKM19]{JKM19}
John {Jasper}, Emily~J. {King}, and Dustin~G. {Mixon}.
\newblock {Game of Sloanes: Best known packings in complex projective space}.
\newblock {\em arXiv e-prints}, page arXiv:1907.07848, Jul 2019.

\bibitem[KP11]{KP11}
N.~O. Kotelina and A.~B. Pevnyi.
\newblock The {V}enkov inequality with weights and weighted spherical
  half-designs.
\newblock volume 173, pages 674--682. 2011.
\newblock Problems in mathematical analysis. No. 55.

\bibitem[LT09]{LT09}
Gustav~I. Lehrer and Donald~E. Taylor.
\newblock {\em Unitary reflection groups}, volume~20 of {\em Australian
  Mathematical Society Lecture Series}.
\newblock Cambridge University Press, Cambridge, 2009.

\bibitem[MP19]{MP19}
Dustin~G. {Mixon} and Hans {Parshall}.
\newblock {Exact Line Packings from Numerical Solutions}.
\newblock {\em arXiv e-prints}, page arXiv:1902.00552, Jan 2019.

\bibitem[MW19]{MW19}
Mozhgan Mohammadpour and Shayne Waldron.
\newblock Complex spherical designs from group orbits.
\newblock preprint, 8 2019.

\bibitem[PW02]{PW02}
Irine Peng and Shayne Waldron.
\newblock Signed frames and {H}adamard products of {G}ram matrices.
\newblock {\em Linear Algebra Appl.}, 347:131--157, 2002.

\bibitem[RS07]{RS07}
Aidan Roy and A.~J. Scott.
\newblock Weighted complex projective 2-designs from bases: optimal state
  determination by orthogonal measurements.
\newblock {\em J. Math. Phys.}, 48(7):072110, 24, 2007.

\bibitem[RS14]{RS14}
Aidan Roy and Sho Suda.
\newblock Complex spherical designs and codes.
\newblock {\em J. Combin. Des.}, 22(3):105--148, 2014.

\bibitem[Rud80]{R80}
Walter Rudin.
\newblock {\em Function theory in the unit ball of {${\bf C}^{n}$}}, volume 241
  of {\em Grundlehren der Mathematischen Wissenschaften [Fundamental Principles
  of Mathematical Science]}.
\newblock Springer-Verlag, New York-Berlin, 1980.

\bibitem[Sal94]{S94}
Attila Sali.
\newblock On the rigidity of spherical {$t$}-designs that are orbits of finite
  reflection groups.
\newblock {\em Des. Codes Cryptogr.}, 4(2):157--170, 1994.

\bibitem[Ser77]{S77}
Jean-Pierre Serre.
\newblock {\em Linear representations of finite groups}.
\newblock Springer-Verlag, New York-Heidelberg, 1977.
\newblock Translated from the second French edition by Leonard L. Scott,
  Graduate Texts in Mathematics, Vol. 42.

\bibitem[SHC03]{CHS03}
N.J.A. Sloane, R.H. Hardin, and P.~Cara.
\newblock Spherical designs in four dimensions.
\newblock In {\em Proc.\ 2003 IEEE Information Theory Workshop}. IEEE, 2003.

\bibitem[ST54]{ST54}
G.~C. Shephard and J.~A. Todd.
\newblock Finite unitary reflection groups.
\newblock {\em Canadian J. Math.}, 6:274--304, 1954.

\bibitem[SZ84]{SZ84}
P.~D. Seymour and Thomas Zaslavsky.
\newblock Averaging sets: a generalization of mean values and spherical
  designs.
\newblock {\em Adv. in Math.}, 52(3):213--240, 1984.

\bibitem[Via17]{V17}
Maryna~S. Viazovska.
\newblock The sphere packing problem in dimension 8.
\newblock {\em Ann. of Math. (2)}, 185(3):991--1015, 2017.

\bibitem[Wal17]{W17}
Shayne Waldron.
\newblock A sharpening of the {W}elch bounds and the existence of real and
  complex spherical {$t$}-designs.
\newblock {\em IEEE Trans. Inform. Theory}, 63(11):6849--6857, 2017.

\bibitem[Wal18]{W18}
Shayne F.~D. Waldron.
\newblock {\em An introduction to finite tight frames}.
\newblock Applied and Numerical Harmonic Analysis. Birkh\"{a}user/Springer, New
  York, 2018.

\bibitem[Wal19]{W19}
Shayne Waldron.
\newblock Spherical designs and their {Gramian}.
\newblock preprint, 9 2019.

\end{thebibliography}
\nocite{*}
\vfil\eject

\end{document}